\documentclass[12pt,reqno]{article}

\usepackage[usenames]{color}
\usepackage{amssymb}
\usepackage{amsmath}
\usepackage{amsthm}
\usepackage{enumerate}

\usepackage[colorlinks=true,
linkcolor=webgreen,
filecolor=webbrown,
citecolor=webgreen]{hyperref}

\AddToHook{package/hyperref/after}{}

\definecolor{webgreen}{rgb}{0,.5,0}
\definecolor{webbrown}{rgb}{.6,0,0}

\newcommand{\seqnum}[1]{\href{https://oeis.org/#1}{\rm \underline{#1}}}

\begin{document}
	
	\theoremstyle{plain}
	\newtheorem{theorem}{Theorem}
	\newtheorem{corollary}[theorem]{Corollary}
	\newtheorem{lemma}[theorem]{Lemma}
	\newtheorem{proposition}[theorem]{Proposition}
	
	\theoremstyle{definition}
	\newtheorem{definition}[theorem]{Definition}
	\newtheorem{example}[theorem]{Example}
	\newtheorem{conjecture}[theorem]{Conjecture}
	
	\theoremstyle{remark}
	\newtheorem{remark}[theorem]{Remark}
	
	\def\Z{\mathbb Z}
	\def\Q{\mathbb Q}
	\def\R{\mathbb R}
	\def\C{\mathbb C}
	
	\begin{center}
		\vskip 1cm{\LARGE\bf Realizability of Some Combinatorial Sequences
		}
		\vskip 1cm
		\large
		Geng-Rui Zhang\\
		School of Mathematical Sciences \\
		Peking University \\
		Beijing 10871\\
		People's Republic of China\\
		\href{mailto:grzhang@stu.pku.edu.cn}{\tt grzhang@stu.pku.edu.cn} \\
	\end{center}
	
	\vskip .2 in
	\begin{abstract}
		A sequence $a=(a_n)_{n=1}^\infty$ of non-negative integers is called realizable if there is a self-map $T:X\to X$ on a set $X$ such that $a_n$ is equal to the number of periodic points of $T$ in $X$ of (not necessarily exact) period $n$, for all $n\geq1$. The sequence $a$ is called almost realizable if there exists a positive integer $m$ such that $(ma_n)_{n=1}^\infty$ is realizable. In this article, we show that certain wide classes of integer sequences are realizable, which contain many famous combinatorial sequences, such as the sequences of Ap\'ery numbers of both kinds, central Delannoy numbers, Franel numbers, Domb numbers, Zagier numbers, and central trinomial coefficients. We also show that the sequences of Catalan numbers, Motzkin numbers, and large and small Schr\"oder numbers are not almost realizable.
	\end{abstract}

\section{Introduction}\label{sec1}
Let $\Z_{\geq0}$ denote the set of non-negative integers and $\Z^+$ the set of positive integers. In this paper, the serial numbers ``Axxxxxx'' associated with certain sequences in the paper all refer to the corresponding sequence numbers in the On-Line Encyclopedia of Integer Sequences (OEIS) \cite{OEIS}. We consider a property for sequences of non-negative integers which is inspired
from dynamical systems.
\begin{definition}\label{Def1.1}
	A sequence $a=(a(n))_{n=1}^\infty$ of non-negative integers is called {\it realizable} if one of the two following equivalent conditions holds.
	\begin{enumerate}[(1)]
		\item \label{realizable1} there is a self-map $T:X\to X$ on a set $X$ such that $$a(n)=\#\{x\in X:T^n(x)=x\}$$ for all $n\geq1$;
		\item \label{realizable2} $(\mu*a)(n)$ is a non-negative integer divisible by $n$, for all $n\geq1$, where $\mu$ is the (classical) M\"obius function and the operation $*$ is the Dirichlet convolution of arithmetic functions.
	\end{enumerate}
	In this case, we also say that the sequence $a=(a(n))_{n=1}^\infty$ is realized via the map $T$ and $T$ realizes $a$.
\end{definition}
\begin{remark}For the equivalence, see \cite[p.\ 398]{PW}. 
	In the description (\ref{realizable2}), the condition 
	$$(\mu*a)(n)\geq0$$ for all $n\geq1$ is called the {\it sign condition}, and the condition 
	$$(\mu*a)(n)\equiv0 \pmod{n}$$
	for all $n\geq1$ is called the {\it Dold condition}. In the description (\ref{realizable1}), by \cite{W} one can equally require $X$ to be an annulus and $f$ a $C^\infty$ diffeomorphism of $X$.
\end{remark}
Miska and Ward \cite{MiW} defined the following generalization of realizability:
\begin{definition}\label{Def1.3}
	Let $a=(a_n)_{n=1}^\infty$ be a sequence of non-negative integers. If there exists $m\in\Z^+$ such that the sequence $(ma_n)_{n=1}^\infty$ is realizable, then we say that $a$ is {\it almost realizable}, and the minimal such $m\in\Z^+$ is denoted by ${\rm Fail}(a)$. When $a$ is not almost realizable, we set ${\rm Fail}(a)=\infty$.
\end{definition}
\begin{example}\label{Eg1.4}We give some examples of sequences which are realizable, almost realizable, or not almost realizable.
	\begin{enumerate}[(1)]
		\item The sequence $(2^n-1)_{n=1}^\infty$ \seqnum{A000225} is realized via the map $$T:\R/\Z\to\R/\Z, x\mapsto 2x\bmod{1},$$ the times-$2$ map on the circle $\R/\Z$.
		\item The sequence $(\vert(-2)^n-1\vert)_{n=1}^\infty$ \seqnum{A062510} is realized via the map $z\mapsto z^{-2}$ on the circle $\mathbb{S}^1=\{z\in\C:\vert z\vert=1\}$.
		\item \label{Eg3} Let $d\in\Z^+$, $X=\{0,1,\ldots,d-1\}^{\Z}$, and $T:X\ni (x_n)_n\mapsto (x_{n+1})_n\in X$ be the shift map on $X$. Then the sequence $(d^n)_{n=1}^\infty$ is realized via $T$.
		\item Consider a sequence $(U_n)_{n=1}^\infty$ given by $U_{n+2}=U_{n+1}+U_{n},n\geq1,U_1=a,U_2=b$, where $a,b\in\Z_{\geq0}$. Note that $b=3a=3$ gives the {\it Lucas sequence} $(L_n)$ (\seqnum{A000032}) and $a=b=1$ gives the {\it Fibonacci sequence} $(F_n)$ (\seqnum{A000045}). Puri and Ward \cite{PW} have showed that the sequence $(U_n)$ is realizable if and only if $b=3a$ if and only if $(U_n)$ is a non-negative integer multiple of the Lucas sequence. In particular, the Fibonacci sequence $(F_n)$ is not realizable. Indeed, Moss and Ward \cite{MoW} proved that $(F_n)$ is not almost realizable, but the sequence $(F_{n^2})_{n=1}^\infty$ (\seqnum{A054783}) is almost realizable with ${\rm Fail}((F_{n^2}))=5$.
		\item Write $S^{(1)}(n,k)$ for the {\it (signless) Stirling number of the first kind}, and $S^{(2)}(n,k)$ for the {\it Stirling number of the second kind}, where $n\geq k\in\Z_{\geq0}$, see \cite[pp.\ 32, 81]{Stanley1}. For $k\geq1$, set $S^{(1)}_k=(S^{(1)}(n+k-1,k))_{n=1}^\infty$ and $S^{(2)}_k=(S^{(2)}(n+k-1,k))_{n=1}^\infty$. Miska and Ward \cite{MiW} have showed that $S^{(1)}_k$ is not almost realizable for every $k\geq1$, and that $S^{(2)}_k$ is realizable if and only if $k\in\{1,2\}$, while $S^{(2)}_k$ is almost realizable with ${\rm Fail}(S^{(2)}_k)\mid(k-1)!$ for every $k\geq1$.
		\item Write $(E_n)$ for the sequence of {\it Euler numbers} \seqnum{A122045}, given by the exponential generating function $$\frac{2}{e^t+e^{-t}}=\sum_{n=0}^\infty E_n\frac{t^n}{n!}.$$Moss \cite[Theorem 5.3.2]{M} proved that the sequence $((-1)^n E_{2n})_{n=1}^\infty$ (\seqnum{A000364}) is realizable.
		\item \label{Eg7} Write $(B_n)$ for the sequence of {\it Bernoulli numbers}, given by the exponential generating function $$\frac{t}{e^t-1}=\sum_{n=0}^\infty B_n\frac{t^n}{n!},$$ cf. \cite[Exercise 5.55]{Stanley}. For each $n\geq1$, let $b_n\in\Z^+$ be the denominator of $B_{2n}$ in the fraction in lowest terms (\seqnum{A002445}) and write $\vert \frac{B_{2n}}{2n}\vert=\frac{\tau_n}{\eta_n}$ with $\tau_n,\eta_n\in\Z^+$, and $\gcd(\tau_n,\eta_n)=1$. Then $(b_n)_{n=1}^\infty,(\tau_n)_{n=1}^\infty$ (\seqnum{A001067} taking absolute value), and $(\eta_n)_{n=1}^\infty$ (\seqnum{A006953}) are realizable according to \cite[Theorem 2.6]{EPPW} and \cite[Theorem 5.5.3, Theorem 5.5.10]{M}. Indeed, the sequence $(\eta_n)_{n=1}^\infty$ can be realized via an endomorphism of a group.
		\item \label{Eg8} Write $(G_n)$ for the sequence of {\it Genocchi numbers} \seqnum{A226158}, given by the exponential generating function $$\frac{-2t}{e^{-t}+1}=\sum_{n=1}^\infty G_n\frac{t^n}{n!}.$$ See \cite[Exercise 5.8]{Stanley} for more on these numbers. It is easy to see that $G_{2n+1}=0$ and $e_n:=(-1)^nG_{2n}=(-1)^n2(1-4^n)B_{2n}$ is a positive odd integer for $n\geq1$. Since $\eta_1=12$ and $\tau_1=1$, by (\ref{Eg7}) and Fermat's little theorem, for every prime $p\geq5$, we have $$e_p=2(4^p-1)\times2p\times\frac{\tau_p}{\eta_p}\equiv4p(4-1)\frac{1}{12}\equiv0\not\equiv1=e_1\pmod{p}.$$If the sequence $(e_n)_{n=1}^\infty$ is almost realizable, then $p\mid {\rm Fail}((e_n))$ for every prime $p\geq5$, contradicting the fact that $\Z^+\ni{\rm Fail}((e_n))<\infty$. Therefore, the sequence $(e_n)_{n=1}^\infty$ is not almost realizable.
		\item Consider a multiplicative arithmetic function $f$ whose values are non-negative integers. Then $\mu*f$ is also multiplicative. Using multiplicativity, we see that $(f(n))_{n=1}^\infty$ is realizable if and only if $$0\leq f(p^m)-f(p^{m-1})=(\mu*f)(p^m)\equiv0\pmod{p^m}$$ for every prime $p$ and $m\in\Z^+$. For example, consider the {\it divisor function} $\sigma_k(n)=\sum_{d\mid n}d^k,n\in\Z^+$, for $k\in\Z^+$. As $\sigma_k$ is multiplicative, from the above observation it is clear that $(\sigma_k(n))_{n=1}^\infty$ is realizable.
		\item Let $({\rm Bell}(n))_{n=0}^\infty$ be the sequence of {\it Bell numbers} \seqnum{A000110}, which is introduced by Bell \cite{Bell}. By the Touchard congruence \cite{T}, for every prime $p$, we have $${\rm Bell}(p)-{\rm Bell}(1)\equiv {\rm Bell}(0)=1\pmod{p}.$$ As in (\ref{Eg8}), we see that the sequence $({\rm Bell}(n))_{n=1}^\infty$ is not almost realizable. For more information on Bell numbers, see \cite{Bell,BR,AKK} and the references therein.
		\item Let $d_n$ be the {\it number of derangements} \seqnum{A000166} of the set $\{1,\ldots,n\}$ for $n\geq1$. It is well-known that $d_n=n!\sum_{k=0}^n\frac{(-1)^k}{k!}$. Then, for every prime number $p$, we have $$d_p-d_1=p!\sum_{k=0}^{p-1}\frac{(-1)^k}{k!}+(-1)^p-0\equiv(-1)^p\pmod{p}.$$ As in (\ref{Eg8}), we see that the sequence $(d_n)_{n=1}^\infty$ is not almost realizable.
	\end{enumerate}
\end{example}
In this article, we consider the realizability of some combinatorial sequences related to binomial coefficients, which are defined below.
\begin{definition}\label{Def1.5}
	For $r\in\Z^+$ and $n,s\in\Z_{\geq0}$, define $$A(n,r,s)=\sum_{k=0}^n \binom{n}{k}^r\binom{n+k}{k}^s.$$
\end{definition}
\begin{remark}\label{Rem1.6}Definition \ref{Def1.5} includes many well-known sequences in combinatorics. For example, $(A(n))_{n=0}^\infty:=(A(n,2,2))_{n=0}^\infty$ is the sequence of {\it Ap\'ery numbers (of the first kind)} \seqnum{A005259}, and $(\beta(n))_{n=0}^\infty:=(A(n,2,1))_{n=0}^\infty$ is the sequence of {\it Ap\'ery numbers of the second kind} \seqnum{A005258}. The Ap\'ery numbers of both kinds were introduced by Ap\'ery \cite{Ap} to prove the irrationality of $\zeta(3)$. For more information on Ap\'ery numbers, see \cite{Beu,WS,Schmidt} and the references therein. The sequence of {\it central Delannoy numbers} \seqnum{A001850} is $(D(n))_{n=0}^\infty:=(A(n,1,1))_{n=0}^\infty$, cf. \cite[Example 6.3.8]{Stanley}. The number $D(n)$ equals the number of paths from the southwest corner $(0, 0)$ of a square grid to the northeast corner $(n,n)$, using only single steps north, northeast, or east. See \cite{BaSc,Sul} to learn more about the central Delannoy numbers. Usually, the sequence $(f^{(3)}(n))$ \seqnum{A000172} is called the sequence of {\it Franel numbers}, which was first introduced by Franel \cite{Franel}. Generally, for $r\in\Z^+$, we call $(f^{(r)}(n))_{n=0}^\infty:=(A(n,r,0))_{n=0}^\infty$ the sequence of {\it Franel numbers of order $r$}. Note that $(f^{(1)}(n))=(2^n)$ has been studied in (\ref{Eg3}) of Example \ref{Eg1.4}, and that $(f^{(2)}(n))$ is the sequence of {\it central binomial coefficients} \seqnum{A000984}.
\end{remark}
\begin{definition}\label{Def1.7}For $r\in\Z^+$ and $n,s,t\in\Z_{\geq0}$, define $$D(n,r,s,t)=\sum_{k=0}^n \binom{n}{k}^r\binom{2k}{k}^s\binom{2(n-k)}{n-k}^t.$$
\end{definition}
\begin{remark}\label{Rem1.8}
	The sequence of {\it Domb numbers} \seqnum{A002895} is $({\rm Domb}(n))_{n=0}^\infty:=(D(n,2,1,1))_{n=0}^\infty$, which is introduced by Domb \cite{Domb}. For more information on Domb numbers, we refer the reader to \cite{Domb,CCL,MaoS} and the references therein. We call $(Z(n))_{n=0}^\infty:=(D(n,1,1,1))_{n=0}^\infty$ the sequence of {\it Zagier numbers} (\seqnum{A081085}), which corresponds to Zagier \cite[E.\ in Table 2]{Zagier}.
\end{remark}
\begin{definition}\label{Def1.9}
	The sequence $(P(n))_{n=0}^\infty$ of {\it Catalan-Larcombe-French numbers} \seqnum{A053175} is given by the formula $$P(n)=2^n\sum_{k=0}^n \binom{n}{2k}\binom{2k}{k}^2 4^{n-2k},n\geq0.$$
\end{definition}
\begin{remark}\label{Rem1.10}
	Catalan \cite{C} showed that $(P(n))$ has the following recurrence relation: $$(n+1)^2P(n+1)=8(3n^2+3n+1)P(n)-128n^2P(n-1),n\geq1.$$ The number $P(n)$ is the ``other'' Catalan number in the sense of Larcombe-French \cite{LF1}. Larcombe and French \cite{LF1} showed that $P(n)$ can be given by elliptic integrals. We will show that $(P(n))_{n=1}^\infty$ is realizable (Remark \ref{Rem1.16}), while the sequence $(C(n))_{n=1}^\infty$ of (true) Catalan numbers is not almost realizable ((\ref{Catalan}) of Theorem \ref{Thm1.20}). 
\end{remark}
\begin{definition}\label{Def1.11}
	For $r\in\Z^+$ and $n,s,t,u\in\Z_{\geq0}$, define $$T(n,r,s,t,u)=\sum_{k=0}^n\binom{n}{2k}^r\binom{n+k}{k}^s\binom{2k}{k}^t\binom{2(n-k)}{n-k}^u.$$
\end{definition}
\begin{remark}\label{Rem1.12}
	For $n\geq1$, the {\it central trinomial coefficient} $T(n)$ is defined to be the coefficient of $x^n$ in $(x^2+x+1)^n$ (\seqnum{A002426}). Clearly, $T(n)=T(n,1,0,1,0)$.
\end{remark}
The sequences given in Definitions \ref{Def1.5}, \ref{Def1.7}, \ref{Def1.9}, and \ref{Def1.11} are all realizable. In fact, the results can be generalized further, see Remarks \ref{Rem3.2} and \ref{Rem3.3}.
\begin{theorem}\label{Thm1.13}For every $s\in\Z_{\geq0}$ and $r\in\Z^+$, the sequence $ (A(n,r,s))_{n=1}^\infty$ is realizable. 
\end{theorem}
\begin{remark}\label{Rem1.14}
	In particular, the sequence of Ap\'ery numbers (of the first kind), the sequence of Ap\'ery numbers of the second kind, the sequence of central Delannoy numbers, and the sequence of Franel numbers of order $r\in\Z^+$ are all realizable.
\end{remark}
\begin{theorem}\label{Thm1.15}
	For every $s,t\in\Z_{\geq0}$, and $r\in\Z^+$, the sequence $ (D(n,r,s,t))_{n=1}^\infty$ is realizable.
\end{theorem}
\begin{remark}\label{Rem1.16}
	In particular, the sequence of Domb numbers and the sequence of Zagier numbers are both realizable. Moreover, the sequence of Catalan-Larcombe-French numbers is also realizable. Larcombe and French \cite[Theorem 3]{LF2} proved that $P(n)=2^n Z(n)$ for all $n\geq0$. As in (\ref{Eg3}) of Example \ref{Eg1.4}, the sequence $(2^n)_{n=1}^{\infty}$ is realized via the shift map $T:X=\{0,1\}^\Z\to X$. By Theorem \ref{Thm1.15}, there exists a map $F:Y\to Y$ on a set $Y$ such that $$Z(n)=\#\{y\in Y:F^n(y)=y\},n\geq1.$$Clearly, the sequence $(P(n))_{n=1}^\infty=(2^n Z(n))_{n=1}^\infty$ is realized via the map $$T\times F:X\times Y\to X\times Y,(x,y)\mapsto(T(x),F(y)).$$
\end{remark}
\begin{remark}\label{Rem1.17}
	By Theorem \ref{Thm1.13}, for all $m,r\in\Z^+$, $s\in\Z_{\geq0}$, and every prime $p$, we have $$A(p^m,r,s)\equiv A(p^{m-1},r,s)\pmod{p^m}.$$In fact, in the proof of Theorem \ref{Thm1.13}, we will show that for every $m,n,r\in\Z^+$, $s\in\Z_{\geq0}$, and every prime $p$, we have
	\begin{equation}\label{1.1}
		A(np^m,r,s)\equiv A(np^{m-1},r,s)\pmod{p^m}.
	\end{equation}
	(The congruence \eqref{1.1} is a result of Theorem \ref{Thm1.13} and Lemma \ref{lemDold}.)
	For certain values of $r$ and $s$, results stronger than \eqref{1.1} have been proved. For example, Straub \cite[Theorem 1.1]{S} asserted that for $(V(n))$ in the $15$ known sporadic Ap\'ery-like sequences \cite[pp.\ 1--2]{S}, arbitrary prime $p\geq3$, and $m,n\in\Z^+$, we have
	\begin{equation}\label{1.2}
		V(np^m)\equiv V(np^{m-1})\pmod{p^{2m}}.
	\end{equation}
	The 15 known sporadic Ap\'ery-like sequences include $(f^{(3)}(n))_n$, $(f^{(4)}(n))_n$ (the sequence \seqnum{A005260}), $(A(n))_n$, $(\beta(n))_n$, $(Z(n))_n$, $({\rm Domb}(n))_n$, $(D(n,2,1,0))_n$ (the sequence \seqnum{A002893}), and some other sequences, see \cite{MS}. Similarly, in the proof of Theorem \ref{Thm1.15}, we will show that for all $m,n,r\in\Z^+$, $s,t\in\Z_{\geq0}$, and every prime $p$, we have
	\begin{equation}\label{1.3}
		D(np^m,r,s,t)\equiv D(np^{m-1},r,s,t)\pmod{p^m}.
	\end{equation}
	(The congruence \eqref{1.3} is a result of Theorem \ref{Thm1.15} and Lemma \ref{lemDold}.)
	Also, for restricted values of $s,t$, and $p$, results stronger than \eqref{1.3} have been proved. For example, Osburn and Sahu \cite[Theorem 1.1]{OS} showed that for all $m,n,s,t\in\Z^+$, $r\in\{2,3,4,\ldots\}$, and every prime $p\geq5$, we have
	\begin{equation}\label{1.4}
		D(np^m,r,s,t)\equiv D(np^{m-1},r,s,t)\pmod{p^{3m}}.
	\end{equation}
\end{remark}
\begin{theorem}\label{Thm1.18}
	For every $r\in\Z^+$ and $s,t,u\in\Z_{\geq0}$, the sequence $(T(n,r,s,t,u))_{n=1}^\infty$ is realizable.
\end{theorem}
Next we consider some other famous sequences involving binomial coefficients as well, but they are not almost realizable. 
\begin{definition}\label{Def1.19}
	\begin{enumerate}[(1)]
		\item The sequence of {\it Catalan numbers} \seqnum{A000108} $(C(n))_{n=0}^\infty$ is given by the formula $$C(n):=\frac{1}{n+1}\binom{2n}{n}=\binom{2n}{n}-\binom{2n}{n+1},n\geq0;$$
		\item The sequence of {\it Motzkin numbers} \seqnum{A001006} $(M(n))_{n=0}^\infty$ is given by the formula $$M(n):=\sum_{k=0}^n \binom{n}{2k} C(k),n\geq0;$$
		\item The sequence of {\it large Schr\"oder numbers} \seqnum{A006318} is given by the formula $$S(n):=\sum_{k=0}^n \binom{n+k}{2k} C(k),n\geq0.$$
	\end{enumerate}
\end{definition}
\begin{remark}
	The Motzkin numbers are first appeared in Motzkin \cite{Motzkin} in a circle chording setting. The {\it large Schr\"oder number} $S(n)$ describes the number of paths from the southwest corner $(0, 0)$ of a square grid to the northeast corner $(n,n)$, using only single steps north, northeast, or east, that do not rise above the SW-NE diagonal.
\end{remark}
\begin{theorem}\label{Thm1.20} The following sequences are not almost realizable.
	\begin{enumerate}[(1)]
		\item \label{Catalan} $(C(n))_{n=1}^\infty$;
		\item $(M(n))_{n=1}^\infty$;
		\item \label{largeS} $(S(n))_{n=1}^\infty$.
	\end{enumerate}
\end{theorem}
\begin{remark}\label{Rem1.21}
	The sequence of {\it little Schr\"oder numbers} \seqnum{A001003} $(s(n))_{n=1}^\infty$ is given by the formula $$s(n):=\sum_{k=1}^{n}N(n,k)2^{k-1},n\geq1,$$ cf. \cite[p.\ 178]{Stanley}. Here $$N(n,k)=\frac{1}{n}\binom{n}{k}\binom{n}{k-1}\in\Z^+$$ is the {\it Narayana number} \seqnum{A001263}, cf. \cite[Exercise 6.36]{Stanley}. It is well-known that $S(n)=2s(n)$ for every $n\geq1$, so from (\ref{largeS}) of Theorem \ref{Thm1.20}, we deduce that the sequence $(s(n))_{n=1}^\infty$ of little Schr\"oder numbers is not almost realizable as well.
\end{remark}
It is interesting to consider the realizability of combinatorial sequences of different types, which may not involve binomial coefficients. Motivated by computations, we make the following conjecture.
\begin{conjecture}\label{Conj1.22}
	For $n\geq1$, let $p(n)$ be the number of partitions of $n$, i.e., ways of writing $n$ as an (unordered) sum of positive integers (\seqnum{A000041}). Then the sequence $$(p(n))_{n=1}^\infty=(1,2,3,5,7,11,15,\ldots)$$ of partition numbers is not almost realizable.
\end{conjecture}
\begin{remark}\label{Rem1.23}However, we show that the sequence $(p(n))_{n=1}^\infty$ satisfies the sign condition. Clearly, the sequence $(p(n))_{n=1}^\infty$ is increasing. By Remark \ref{Rem2.5}, it suffices to check that
	\begin{equation}\label{1.5}
		p(2n)\geq np(n),n\geq1.
	\end{equation}
	By \cite[Corollary 3.1]{Ma} and \cite[Theorem 15.7]{LW}, for all integers $n\geq3$, we have $$\frac{1}{14}e^{2\sqrt{n}}<p(n)<\frac{\pi}{\sqrt{6(n-1)}}e^{\pi\sqrt{\frac{2}{3}n}}.$$ Thus, for every integer $n\geq523$, we have$$\frac{p(2n)}{p(n)}>\frac{\sqrt{6(n-1)}}{14\pi}e^{(2\sqrt{2}-\pi\sqrt{\frac{2}{3}})\sqrt{n}}>n.$$ For integers $1\leq n\leq522$, the equation \eqref{1.5} has been checked by direct computation via a computer. Hence the sequence $(p(n))_{n=1}^\infty$ satisfies the sign condition.
\end{remark}
In \S \ref{sec2}, we consider several useful lemmas. In \S \ref{sec3}, we complete the proofs of Theorems \ref{Thm1.13}, \ref{Thm1.15}, \ref{Thm1.18}, and \ref{Thm1.20}, and give some remarks.

\section{Auxiliary lemmas}\label{sec2}
First, we need the following well-known theorem of Kummer \cite{K}.
\begin{lemma}[Kummer's theorem] \label{Lem2.1}Given a prime number $p$ and integers $n\geq m\geq0$, the value of $\nu_p\left( \binom{n}{m}\right) $ is equal to the number of carries when $m$ is added to $n-m$ in base $p$. Here $\nu_p$ denotes the standard $p$-adic valuation on $\Q$.
\end{lemma}
The following lemma is \cite[Corollary of p.\ 490]{CT} by Helou and Terjanian.
\begin{lemma}\label{Lem2.2}
	Let $n\geq m\geq0$ be integers.
	\begin{enumerate}[(1)]
		\item For every prime $p\geq5$, we have $$\binom{np}{mp}\equiv\binom{n}{m}\pmod{p^{3+\max\left(\nu_p(m),\nu_p(n-m)\right)+\nu_p\left(\binom{n}{m}\right)}}.$$
		\item For $p=3$, we have $$\binom{3n}{3m}\equiv\binom{n}{m}\pmod{3^{2+\max\left(\nu_3(m),\nu_3(n-m)\right)+\nu_3\left(\binom{n}{m}\right)}}.$$
		\item For $p=2$, we have $$\binom{2n}{2m}\equiv\binom{n}{m}\pmod{2^{1+\max\left(\nu_2(m),\nu_2(n-m)\right)+\nu_2\left(\binom{n}{m}\right)}}.$$
	\end{enumerate}
	Here, for two integers $A,B$, and a prime $p$, the expression $A\equiv B\pmod{p^\infty}$ means that $A\equiv B\pmod{p^N}$ for all $N\in\Z^+$, or equivalently, $A=B$.
\end{lemma}
Note that in Lemma \ref{Lem2.2}, the conclusion trivially holds when $m=0$.
\begin{lemma}\label{Lem2.3}Let $n\in\Z_{\geq0}$ and $p$ be a prime divisor of $n$. Set $m=\nu_p(n)\in\Z^+\cup\{\infty\}$. Then $$\binom{n}{\lambda p}\equiv \binom{n/p}{\lambda}\pmod{p^{\max(m,\nu_p(n-\lambda p))}},$$ for every non-negative integer $\lambda$.
\end{lemma}
\begin{proof} When $n=0$, we have $\binom{n}{\lambda p}=0= \binom{n/p}{\lambda}$. Hence the conclusion is clear. We may assume that $n>0$.
	Write $n=l p^m$. Observe that $\gcd(l,m)=1$. If $\lambda\geq l p^{m-1}$ or $\lambda=0$, then $\binom{n}{\lambda p}=\binom{n/p}{\lambda}\in\{0,1\}$, so the conclusion holds. Now assume that $0<\lambda<l p^{m-1}$. As $1+\nu_p((n/p)-\lambda)=\nu_p(n-\lambda p)$, by Lemma \ref{Lem2.2}, it suffices to show that $$1+\max\left(\nu_p(\lambda),\nu_p(l p^{m-1}-\lambda)\right)+\nu_p\left( \binom{l p^{m-1}}{\lambda}\right) \geq m.$$ Write $\lambda=p^t q$, where $t=\nu_p(\lambda)\in\Z_{\geq0}$ and $q\in\Z^+$ with $p\nmid q$. When $t\geq m-1$, we have $$1+\max\left(\nu_p(\lambda),\nu_p(l p^{m-1}-\lambda)\right)+\nu_p\left( \binom{l p^{m-1}}{\lambda}\right) \geq 1+\nu_p(\lambda)=1+t\geq m.$$ If $t\leq m-2$, then $\nu_p(n/p)=m-1>m-2\geq t=\max(\nu_p(\lambda),\nu_p((n/p)-\lambda))$, so by Lemma \ref{Lem2.1} we see that $$\nu_p\left( \binom{l p^{m-1}}{\lambda}\right)=\nu_p\left( \binom{l p^{m-1}}{q p^t}\right)\geq m-1-t.$$ Thus, $$1+\max\left(\nu_p(\lambda),\nu_p(l p^{m-1}-\lambda)\right)+\nu_p\left( \binom{l p^{m-1}}{\lambda}\right) \geq 1+t+m-1-t=m.$$
\end{proof}
The following lemma gives a sufficient condition for the sign condition.
\begin{lemma}\label{Lem2.4}
	Let $(b(n))_{n=1}^\infty$ be a sequence of non-negative real numbers. Assume that there is a constant $C\geq 1.221$ such that $b(n+1)\geq C b(n)$ for every $n\geq 1$. Then $(\mu*b)(n)\geq0$ for every $n\geq1$.
\end{lemma}
\begin{proof}
	Set $f=\mu*b$. By assumption we have $b(n+k)\geq C^k b(n)$ for every $n\geq1$ and $k\geq0$. In particular, $(b(n))_{n=1}^\infty$ is non-decreasing and non-negative. Trivially, $f(1)=b(1)\geq0$. For every prime number $p$ and $k\in\Z^+$, since $p^k>p^{k-1}$, we have $$f(p^k)=b(p^k)-b(p^{k-1})\geq 0.$$ For every pair of distinct prime numbers $p_1\neq p_2$ and $k_1, k_2\in\Z^+$, we have 
	\begin{align*}
		f\left(p_1^{k_1} p_2^{k_2}\right)&=b\left(p_1^{k_1} p_2^{k_2}\right)-b\left(p_1^{k_1-1} p_2^{k_2}\right)-b\left(p_1^{k_1} p_2^{k_2-1}\right)+b\left(p_1^{k_1-1} p_2^{k_2-1}\right)\\
		&\geq b\left(p_1^{k_1} p_2^{k_2}\right) - C^{p_1^{k_1-1} p_2^{k_2}-p_1^{k_1 } p_2^{k_2}}b\left(p_1^{k_1} p_2^{k_2}\right) - C^{p_1^{k_1} p_2^{k_2-1}-p_1^{k_1} p_2^{k_2}}b\left(p_1^{k_1} p_2^{k_2}\right)+0\\
		&\geq \left(1-C^{3-3\times2}-C^{2-3\times2}\right)b\left(p_1^{k_1} p_2^{k_2}\right)\\
		&\geq\left(1-1.221^{-3}-1.221^{-4}\right)b\left(p_1^{k_1} p_2^{k_2}\right)\\
		&\geq 0.
	\end{align*}
	The minimum positive integer with at least three distinct prime divisor is $30=2\times3\times5$, so it suffices to show that $f(n)\geq0$ for $n\geq 30$. Assume that $n\geq 30$, and set $m=\lfloor \frac{n}{2}\rfloor\geq15$. Then
	\begin{align*}
		f(n)&=\sum\limits_{d\mid n} \mu\left(\frac{n}{d}\right)b(d)=b(n)+\sum\limits_{d\mid n,d\neq n} \mu\left(\frac{n}{d}\right)b(d)\geq b(n)-\sum\limits_{d\mid n, d\neq n}b(d)\\
		&\geq b(n)-\sum\limits_{d=1}^{m}b(d)\geq C^mb(m)-\sum\limits_{d=1}^{m}C^{d-m}b(m)\geq (C^m-m)b(m)\\
		&\geq (1.221^m-m)b(m)\geq 0,
	\end{align*}
	since $m\geq 15$.
\end{proof}
\begin{remark}\label{Rem2.5}
	The converse of Lemma \ref{Lem2.4} is false, in general. For example, the trivial sequence $(1)_{n=1}^\infty$ provides a counterexample.
	Note that the number $1.221$ in the statement can be replaced by the unique positive root $x_0=1.220744\cdots$ of the equation $x^4=x+1$. Puri \cite{P} observed that if $(b(n))_{n=1}^{\infty}$ is a non-decreasing sequence of non-negative real numbers with
	\begin{equation}\label{2.6}
		b(2n)\geq nb(n),
	\end{equation}
	for all $n\geq1$, then $(\mu*b)(n)\geq0$ for all $n\geq1$, cf.\ \cite[proof of Lemma 8]{MiW}. Note that Lemma \ref{Lem2.4} cannot be deduced directly from the observation of Puri since $1.221^n<n$ for $n\in\{2,3,\ldots,12\}$, and vice versa.
\end{remark}
We consider the Dold condition now, which has an equivalent statement as follows.
\begin{lemma}\label{lemDold}
	Let $(V(n))_{n=1}^\infty$ be a sequence of integers. Then the following conditions are equivalent.
	\begin{enumerate}[(1)]
		\item \label{lemDold1} for every $n,m\in\Z^+$, and every prime number $p$, we have $$V(np^m)\equiv V(np^{m-1})\pmod{p^m};$$
		\item $(V(n))_{n=1}^\infty$ satisfies the Dold condition, i.e., $(\mu*V)(n)\equiv 0\pmod{n}$, for every $n\geq1$.
	\end{enumerate}
\end{lemma}
\begin{proof}
	Set $g=\mu*V$. Given an arbitrary integer $n\geq2$, write $n=p_1^{m_1}p_2^{m_2}\cdots p_l^{m_l}$, where $p_1,\ldots,p_l$ are pairwise distinct primes, and $l,m_1,\ldots,m_l\in\Z^+$. Set $n_1=p_2^{m_2}\cdots p_l^{m_l}$. Here $n_1=1$ when $l=1$. Then, we have
	\begin{align}
		\begin{aligned}\label{eqn2}
			g(n)&=\sum\limits_{d\mid n} \mu(d)V\left(\frac{n}{d}\right)=\sum\limits_{d\mid n_1} \left( \mu(d)V\left(\frac{n}{d}\right)+\mu(dp_1)V\left(\frac{n}{dp_1}\right)\right) \\
			&=\sum\limits_{d\mid n_1} \mu(d)\left( V\left(\frac{n_1}{d} p_1^{m_1}\right)+\mu(p_1)V\left(\frac{n_1}{d} p_1^{m_1-1}\right)\right)\\
			&=\sum\limits_{d\mid n_1} \mu(d)\left(V\left(\frac{n_1}{d} p_1^{m_1}\right)-V\left(\frac{n_1}{d} p_1^{m_1-1}\right)\right).
		\end{aligned}
	\end{align}
	
	First assume that for every $n,m\in\Z^+$, and every prime number $p$, we have $$V(np^m)\equiv V(np^{m-1})\pmod{p^m}.$$ We show that the Dold condition holds. The congruence $g(1)\equiv 0$ (mod $1$) trivially holds. For an arbitrary integer $n=p_1^{m_1}p_2^{m_2}\cdots p_l^{m_l}\geq2$ as above, we have $$g(n)=\sum\limits_{d\mid n_1} \mu(d)\left(V\left(\frac{n_1}{d} p_1^{m_1}\right)-V\left(\frac{n_1}{d} p_1^{m_1-1}\right)\right)\equiv0\pmod{p_1^m}$$ by \eqref{eqn2}. Similarly, we have $g(n)\equiv0$ (mod $p_j^{m_j}$) for $1< j\leq l$. Thus, we get $g(n)\equiv0$ (mod $n$), i.e., the Dold condition is verified. 
	
	Now assume that for every $n\in\Z^+$, we have $g(n)\equiv0$ (mod $n$). We will show that for every $n,m\in\Z^+$, and every prime number $p$, we have
	\begin{align}\label{eqDold}
		V(np^m)\equiv V(np^{m-1})\pmod{p^m}.
	\end{align} Write $n=n_1p^s$, where $n_1\in\Z^+$ and $s\in\Z_{\geq0}$ with $\gcd(n_1,p)=1$. Note that the congruence $$V(n_1p^{m+s})\equiv V(n_1p^{m+s-1})\pmod{p^{m+s}}$$ implies the congruence \eqref{eqDold}, hence it suffices to show that \eqref{eqDold} holds in the case $\gcd(n,p)=1$. We use induction on $n\in\Z^+$. When $n=1$, the congruence \eqref{eqDold} follows from $$V(p^m)-V(p^{m-1})=g(p^m)\equiv0\pmod{p^m}.$$ Let $n$ be an arbitrary integer at least $2$ and assume that \eqref{eqDold} hold for all smaller $n$ (with $\gcd(n,p)=1$). Thus, by \eqref{eqn2}, the inductive hypothesis, and the Dold condition, we have
	\begin{align*}
		0&\equiv g(np^m)=\sum\limits_{d\mid n} \mu(d)\left(V\left(\frac{n}{d} p^{m}\right)-V\left(\frac{n}{d} p^{m-1}\right)\right)\\
		&\equiv\mu(1)\left(V\left(n p^{m}\right)-V\left(n p^{m-1}\right)\right)+\sum\limits_{d\mid n,d>1}\mu(d)\cdot0\\
		&=V\left(n p^{m}\right)-V\left(n p^{m-1}\right)\pmod{p^{m}}.
	\end{align*}
	(Here $np^m,n,p,m$ correspond to $n,n_1,p_1,m_1$ in \eqref{eqn2}, respectively.) Therefore, we have showed that \eqref{eqDold} holds for every $n,m\in\Z^+$, and every prime $p$.
\end{proof}
\begin{remark}
	From the proof of Lemma \ref{lemDold}, we can require that $n$ and $p$ are coprime in condition (\ref{lemDold1}).
\end{remark}
\section{Proofs of theorems}\label{sec3}
\begin{proof}[Proof of Theorem \ref{Thm1.13}] Fix $s\in\Z_{\geq0}$ and $r\in\Z^+$. Write $V(n)=A(n,r,s)$, $n\geq1$, for simplicity of notation. Set $g=\mu*V$. For $n=1$, it is clear that $g(1)=V(1)=1+2^s$ is divisible by $1$. Given arbitrary $m,n\in\Z^+$, and an arbitrary prime number $p$, we consider the difference $V(np^m)-V(np^{m-1})$ modulo $p^m$. Let $0\leq k\leq n p^m$ be an integer. If $p\nmid k$, then we have $$\binom{n p^m}{k}\equiv 0\pmod{p^m}$$ by Lemma \ref{Lem2.1}. If $p\mid k$, then $$\binom{n p^m}{k}\equiv \binom{n p^{m-1}}{k/p}\pmod{p^{m}},\binom{n p^m+k}{k}\equiv \binom{n p^{m-1}+k/p}{k/p}\pmod{p^{m}}$$ by Lemma \ref{Lem2.3}, since $\nu_p(np^m+k-k)\geq m$. Therefore, 
	\begin{align*}
		V(np^m)-V(np^{m-1})&=\sum_{k=0}^{np^m} \binom{np^m}{k}^r\binom{np^m+k}{k}^s-\sum_{k=0}^{np^{m-1}} \binom{np^{m-1}}{k}^r\binom{np^{m-1}+k}{k}^s\\
		&\equiv0+\sum_{\lambda=0}^{np^{m-1}} \binom{np^m}{\lambda p}^r\binom{np^m+\lambda p}{\lambda p}^s-\sum_{k=0}^{np^{m-1}}\binom{np^{m-1}}{k}^r\binom{np^{m-1}+k}{k}^s\\
		&\equiv\sum_{\lambda=0}^{np^{m-1}} \binom{np^{m-1}}{\lambda}^r\binom{np^{m-1}+\lambda}{\lambda}^s-\sum_{k=0}^{np^{m-1}} \binom{np^{m-1}}{k}^r\binom{np^{m-1}+k}{k}^s\\
		&=0\pmod{p^m}.
	\end{align*}
	We get that $$V(np^m)-V(np^{m-1})\equiv 0\pmod{p^m}.$$ By Lemma \ref{lemDold}, we see that $(V(n))_{n=1}^\infty$ satisfies the Dold condition.
	
	It suffices to show that $g(n)\geq0$ for all $n\in\Z^+$. For every $n\geq1$, we have
	\begin{align*}
		V(n+1)&=\sum_{k=0}^{n+1}\binom{n+1}{k}^r\binom{n+1+k}{k}^s=1+\sum_{k=1}^{n+1}\binom{n+1}{k}^r\binom{n+1+k}{k}^s\\
		&=1+\sum_{k=1}^{n+1}\left( \binom{n}{k}+\binom{n}{k-1}\right) ^r\left( \binom{n+k}{k}+\binom{n+k}{k-1}\right)^s\\
		&\geq 1+\sum_{k=1}^{n+1}\left( \binom{n}{k}^r\binom{n+k}{k}^s+\binom{n}{k-1}^r\binom{n+k}{k-1}^s\right)\\
		&\geq 1+\sum_{k=1}^{n+1}\left(\binom{n}{k}^r\binom{n+k}{k}^s+\binom{n}{k-1}^r\binom{n+k-1}{k-1}^s\right)\\
		&=2\sum_{k=0}^{n}\binom{n}{k}^r\binom{n+k}{k}^s=2V(n).
	\end{align*}
	As $2>1.221$, the Dold condition $g(n)\geq0$ for all $n\in\Z^+$ follows from Lemma \ref{Lem2.4}. Therefore, the sequence $(V(n))_{n=1}^\infty$ is realizable.
\end{proof}
\begin{remark}\label{Rem3.1}
	In the proof of Theorem \ref{Thm1.13}, we have showed that $A(n+1,r,s)\geq2A(n,r,s)$, for all $n,s\geq0$, and $r\geq1$. For some specified values of $r$ and $s$, some stronger results are known. Recall that $A(n)=A(n,2,2),\beta(n)=A(n,2,1),D(n)=A(n,1,1)$, and $f^{(r)}(n)=A(n,r,0)$. Xia and Yao \cite[Corollary 5]{XY} proved that the sequence $(A(n))_{n=0}^\infty$ is {\it strictly log-convex}, i.e., $$\frac{A(n+1)}{A(n)}>\frac{A(n)}{A(n-1)},n\geq1.$$ Hence, for all $n\geq1$, we have $$\frac{A(n+1)}{A(n)}\geq\frac{A(2)}{A(1)}=\frac{73}{5}>2.$$Xia and Yao also showed that the sequence $(D(n))_{n=0}^\infty$ is strictly log-convex \cite[Corollary 6]{XY}, hence we have $$\frac{D(n+1)}{D(n)}\geq\frac{D(2)}{D(1)}=\frac{13}{3}>2,n\geq1.$$ The result \cite[Theorem 5.2]{CX} of Chen and Xia implies that the sequence $(\beta(n))_{n=0}^\infty$ is strictly log-convex, hence we get $$\frac{\beta(n+1)}{\beta(n)}\geq\frac{\beta(2)}{\beta(1)}=\frac{19}{3}>2,n\geq1.$$ The result \cite[Corollary 4.3]{D} of Do\v{s}li\'c implies that for $r\in\{3,4\}$, the sequence $(f^{(r)}(n))_{n=0}^\infty$ is {\it log-convex}, i.e., $$f^{(r)}(n+1)f^{(r)}(n-1)\geq f^{(r)}(n)^2,n\geq1.$$
\end{remark}
\begin{proof}[Proof of Theorem \ref{Thm1.15}]
	Fix $s,t\in\Z_{\geq0}$, and $r\in\Z^+$. Write $W(n)=D(n,r,s,t)$, $n\geq1$, for simplicity of notation. Set $g=\mu*W$. For $n=1$, it is clear that $g(1)=W(1)=2^t+2^s$ is divisible by $1$. Given $m,n\in\Z^+$, and an arbitrary prime number $p$, we consider the difference $A:=W(np^m)-W(np^{m-1})$ modulo $p^m$. Let $0\leq k\leq np^m$ be an integer. If $p\nmid k$, then we have
	\begin{align}\label{3.7}
		\binom{np^m}{k}\equiv 0\pmod{p^m}
	\end{align}
	by Lemma \ref{Lem2.1}. Now assume that $p\mid k$ and write $k=lp^u$, where $(l,u)=(0,m)$ when $k=0$, and $l,u\in\Z^+$ with $\gcd(l,p)=1$ when $k>0$. We claim that
	\begin{align}\label{3.10}
		\binom{np^m}{k}^r\binom{2k}{k}^s\binom{2(np^m-k)}{np^m-k}^t\equiv\binom{np^{m-1}}{lp^{u-1}}^r\binom{2lp^{u-1}}{lp^{u-1}}^s\binom{2(np^{m-1}-lp^{u-1})}{np^{m-1}-lp^{u-1}}^t\pmod{p^m}.
	\end{align}
	Note that Lemma \ref{Lem2.3} implies that $$\binom{np^m}{k}^r\equiv\binom{np^{m-1}}{lp^{u-1}}^r\pmod{p^m}.$$By Lemma \ref{Lem2.1}, we have $\nu_p\left(\binom{2M}{M}\right)\geq\delta_{2,p}$ for all $M\in\Z^+$, where $\delta_{2,p}$ is the Kronecker symbol. From Lemma \ref{Lem2.2}, we see that
	\begin{align}\label{3.11}
		\binom{2k}{k}&\equiv\binom{2lp^{u-1}}{lp^{u-1}}\pmod{p^{1+u}},\\ \binom{2(np^m-k)}{np^m-k}&\equiv\binom{2(np^{m-1}-lp^{u-1})}{np^{m-1}-lp^{u-1}}\pmod{p^{1+\min(u,m)}}.\label{3.12}
	\end{align}
	Thus, the congruence \eqref{3.10} holds if $u\geq m-1$. Assume that $1\leq u\leq m-2$. By the definition of $u$, we have $0<k<np^m$. From Lemma \ref{Lem2.1}, we get that
	\begin{equation}\label{3.13}
		v_p\left(\binom{np^m}{k}\right)=v_p\left(\binom{np^{m-1}}{k/p}\right)\geq m-u.
	\end{equation}
	Set $N=\max(0,m-r(m-u))$. To show \eqref{3.10}, it suffices to prove the congruences 
	\begin{align}\label{3.14}\binom{2k}{k}\equiv\binom{2lp^{u-1}}{lp^{u-1}}\pmod{p^N} \text{ and } \binom{2(np^m-k)}{np^m-k}\equiv\binom{2(np^{m-1}-lp^{u-1})}{np^{m-1}-lp^{u-1}}\pmod{p^N},
	\end{align}
	which follows from \eqref{3.11} and \eqref{3.12}, since $r\geq1$ and $1\leq u\leq m-2$ imply $$N=\max(0,ru-(r-1)m)< 1+u=1+\min(u,m).$$By \eqref{3.7} and \eqref{3.10}, the difference $A=W(np^m)-W(np^{m-1})$ satisfies
	\begin{align*}
		A&=\sum_{k=0}^{np^m} \binom{np^m}{k}^r\binom{2k}{k}^s\binom{2(np^m-k)}{np^m-k}^t-\sum_{k=0}^{np^{m-1}} \binom{np^{m-1}}{k}^r\binom{2k}{k}^s\binom{2(np^{m-1}-k)}{np^{m-1}-k}^t\\
		&\equiv0+\sum_{\lambda=0}^{np^{m-1}} \binom{np^m}{\lambda p}^r\binom{2\lambda p}{\lambda p}^s\binom{2(np^m-\lambda p)}{np^m-\lambda p}^t-\sum_{k=0}^{np^{m-1}} \binom{np^{m-1}}{k}^r\binom{2k}{k}^s\binom{2(np^{m-1}-k)}{np^{m-1}-k}^t\\
		&\equiv\sum_{\lambda=0}^{np^{m-1}} \binom{np^{m-1}}{\lambda}^r\binom{2\lambda}{\lambda}^s\binom{2(np^{m-1}-\lambda)}{np^{m-1}-\lambda}^t-\sum_{k=0}^{np^{m-1}}\binom{np^{m-1}}{k}^r\binom{2k}{k}^s\binom{2(np^{m-1}-k)}{np^{m-1}-k}^t\\
		&=0\pmod{p^m}.
	\end{align*}
	We get $W(np^m)\equiv W(np^{m-1})
	\pmod{p^m}$. Hence, the Dold condition for $(W(n))_{n=1}^{\infty}$ follows from Lemma \ref{lemDold}.
	
	We consider the sign condition. It is clear that the sequence $\left(\binom{2n}{n}\right)_{n=0}^\infty$ is strictly increasing by looking the ratios of two adjacent terms. For all $n\geq1$, we have
	\begin{align*}
		W(n+1)
		&=\sum_{k=0}^{n+1} \binom{n+1}{k}^r\binom{2k}{k}^s\binom{2(n+1-k)}{n+1-k}^t\\
		&=\binom{2(n+1)}{n+1}^t+\sum_{k=1}^{n+1} \left( \binom{n}{k}+\binom{n}{k-1}\right) ^r\binom{2k}{k}^s\binom{2(n+1-k)}{n+1-k}^t\\
		&\geq\binom{2(n+1)}{n+1}^t+\sum_{k=1}^{n+1} \left( \binom{n}{k}^r+\binom{n}{k-1}^r\right) \binom{2k}{k}^s\binom{2(n+1-k)}{n+1-k}^t\\
		&\geq\binom{2n}{n}^t+\sum_{k=1}^{n} \binom{n}{k}^r \binom{2k}{k}^s\binom{2(n-k)}{n-k}^t+\sum_{k=0}^{n}\binom{n}{k}^r\binom{2k}{k}^s\binom{2(n-k)}{n-k}^t\\
		&=2\sum_{k=0}^{n} \binom{n}{k}^r \binom{2k}{k}^s\binom{2(n-k)}{n-k}^t=2W(n).
	\end{align*}
	As $2>1.221$ and $W(n)>0$ for all $n\geq1$, the sign condition for $(W(n))_{n=1}^{\infty}$ holds by Lemma \ref{Lem2.4}. Thus, the sequence $(W(n))_{n=1}^\infty$ is realizable.
\end{proof}
\begin{remark}\label{Rem3.2}For every $r\in\Z^+$ and $s,t,u\in\Z_{\geq0}$, set $$C(n,r,s,t,u)=\sum_{k=0}^n \binom{n}{k}^r\binom{n+k}{k}^s\binom{2k}{k}^t\binom{2(n-k)}{n-k}^u,n\geq1.$$
	A very similar argument proves that the sequence $(C(n,r,s,t,u))_{n=1}^\infty$ is realizable, which generalizes both Theorem \ref{Thm1.13} and Theorem \ref{Thm1.15}.
\end{remark}
\begin{proof}[Proof of Theorem \ref{Thm1.18}] Fix $r\in\Z^+$ and $s,t,u\in\Z_{\geq0}$. As in Remark \ref{Rem3.2}, applying Lemma \ref{lemDold}, the combination of the proof of Theorem \ref{Thm1.15} and Lemma \ref{lemDold} with small modification implies the Dold condition for $(T(n,r,s,t,u)=:X(n))_{n=1}^\infty$. The only differences are that in many places $k$ should be replaced by $2k$, so the formula (\ref{3.13}) should have a correction term, i.e., (\ref{3.13}) should be $$v_p\left(\binom{np^m}{2k}\right)=v_p\left(\binom{np^{m-1}}{2k/p}\right)\geq m-u-\delta_{2,p},$$ and that $N$ should be $\max(0,m-r(m-u-\delta_{2,p}))$. Since $m-r(m-u-1)\leq 1+u$ under the conditions $r\geq1$ and $1\leq u\leq m-2$, we see that the proof is still valid.
	
	For the sign condition, we first consider the easier case of central trinomial coefficients, i.e., the case when $(r,s,t,u)=(1,0,1,0)$. We will use the observation of Puri stated as in Remark \ref{Rem2.5}. Fix $n\in\Z^+$. For a polynomial $h(x)$ and $k\in\Z_{\geq0}$, let $[x^k]h(x)$ denote the coefficient of $x^k$ in $h(x)$. Then $$T(2n)=[x^{2n}](x^2+x+1)^{2n}\geq([x^n](x^2+x+1)^n)^2=T(n)^2.$$
	Clearly, the sequence of central trinomial coefficients is increasing. By Remark \ref{Rem2.5}, it suffices to show that $T(n)\geq n$, for every $n\geq1$. For $n\in\{1,2,3\}$, $T(n)\geq n$ trivially holds. When $n\geq4$, we have $$T(n)\geq1+\binom{n}{2\lfloor\frac{n}{2}\rfloor}\binom{2\lfloor\frac{n}{2}\rfloor}{\lfloor\frac{n}{2}\rfloor}\geq1+2\left\lfloor\frac{n}{2}\right\rfloor\geq n.$$ Therefore, the sequence $(T(n))_{n=1}^\infty$ satisfies the sign condition.
	
	Now we consider the sign condition for the general case. We will prove that
	\begin{equation}\label{3.15}
		X(n+1)\geq \frac{7}{5}X(n),n\geq1,
	\end{equation}from which the sign condition follows, applying Lemma \ref{Lem2.4}.
	
	Fix $n\in\Z^+$. When $n=1$, we have $X(2)=6^u+3^s 2^{t+u}\geq 2\times2^u=2X(1)$. When $n=2$, we have
	\begin{align*}
		X(3)&=20^u+3^r4^s2^t6^u\geq2\times6^u-\delta_{u,0}+3\times4^s2^t6^u\geq2\times6^u-\delta_{u,0}+1+2\times4^s2^t6^u\\
		&\geq2(6^u+4^s2^t6^u)\geq2(6^u+3^s2^{t+u})=2X(2).
	\end{align*}
	Similarly, it is easy to check that $X(4)\geq2X(3)$ and $X(5)\geq2X(4)$ hold. Now assume that $n\geq5$. (Here $n\geq5$ ensures that $n_1-1\geq1$ below.)
	
	For integers $m,l\geq0$, set $$f_{m}(l)=\binom{m+l}{l}^s\binom{2l}{l}^t,$$ which is non-decreasing with respect to both $m$ and $l$.
	
	We give a lower bound of $X(n+1)-X(n)=:B$ as follows. As in the proofs of Theorems \ref{Thm1.13} and \ref{Thm1.15}, from the identity $\binom{m+1}{l+1}=\binom{m}{l}+\binom{m}{l+1}$, we get $\binom{m+1}{l+1}^v\geq\binom{m}{l}^v+\binom{m}{l+1}^v$, where $m\geq l\geq 0$ and $v\geq1$ are integers. Set $n_0=\lfloor\frac{n+1}{2}\rfloor\in\left[3,n-2\right] $. Since the central binomial coefficient $\binom{2m}{m}$ is strictly increasing and $\binom{n+1+k}{k}\geq\binom{n+k}{k}$, we have
	\begin{align}
		\begin{aligned}\label{eqdiff}
			B&=\binom{2(n+1)}{n+1}^u+\sum_{k=1}^{n_0}\binom{n+1}{2k}^r f_{n+1}(k)\binom{2(n+1-k)}{n+1-k}^u-X(n)\\
			&\geq\binom{2n}{n}^u+\sum_{k=1}^{n_0}\left( \binom{n}{2k}^r+\binom{n}{2k-1}^r\right) f_{n}(k)\binom{2(n+1-k)}{n+1-k}^u-X(n)\\
			&\geq\sum_{k=1}^{n_0}\binom{n}{2k-1}^r f_{n}(k)\binom{2(n+1-k)}{n+1-k}^u.
		\end{aligned}
	\end{align}
	
	Using \eqref{eqdiff}, we will give a lower bound of $X(n+1)-2X(n)$. 
	Set $n_1=\left\lfloor\frac{n+3}{4}\right\rfloor=\lceil\frac{n}{4}\rceil\in\left[ 2,n-3\right]$. When $1\leq k\leq n_1$, we have $$2k-1\leq2\left\lceil\frac{n}{4}\right\rceil-1\leq\frac{n+1}{2}.$$ When $n_1+1\leq k\leq n_0$, we have $$2k-1\geq2\left\lceil\frac{n}{4}\right\rceil+1\geq\frac{n}{2}+1.$$Set
	\begin{align*}
		\Delta^\prime=\left( \binom{n}{2n_1-1}^r f_{n}(n_1)-\binom{n}{2n_1-2}^r f_{n}(n_1-1)\right) \binom{2(n+1-n_1)}{n+1-n_1}^u
	\end{align*}
	and
	\begin{align*}
		\Delta=\binom{n}{2n_1-2}^r \left(f_{n}(n_1)-f_{n}(n_1-1)\right)\binom{2(n+1-n_1)}{n+1-n_1}^u.
	\end{align*}
	Then, we have
	\begin{equation}
		\begin{aligned}\label{eqq3}
			B&\geq\sum_{k=1}^{n_0}\binom{n}{2k-1}^r f_{n}(k)\binom{2(n+1-k)}{n+1-k}^u\\
			&=\sum_{k=1}^{n_1}\binom{n}{2k-1}^r f_{n}(k)\binom{2(n+1-k)}{n+1-k}^u+\sum_{k=n_1+1}^{n_0}\binom{n}{2k-1}^r f_{n}(k)\binom{2(n+1-k)}{n+1-k}^u\\
			&\geq\left(\sum_{k=1}^{n_1}\binom{n}{2k-2}^r f_{n}(k-1)\binom{2(n+1-k)}{n+1-k}^u+\Delta^\prime\right)+\sum_{k=n_1+1}^{n_0}\binom{n}{2k}^r f_{n}(k)\binom{2(n-k)}{n-k}^u\\
			&=\sum_{l=0}^{n_1-1}\binom{n}{2l}^r f_{n}(l)\binom{2(n-l)}{n-l}^u+\Delta^\prime+\sum_{k=n_1+1}^{n_0}\binom{n}{2k}^r f_{n}(k)\binom{2(n-k)}{n-k}^u\\
			&=\sum_{l=0,l\neq n_1}^{n_0} \binom{n}{2l}^r f_{n}(l)\binom{2(n-l)}{n-l}^u+\Delta^\prime,
		\end{aligned}
	\end{equation}
	where we use \eqref{eqdiff} in the first inequality.
	Thus, we have
	\begin{equation}\label{3.19}
		X(n+1)-2X(n)\geq\Delta^\prime-\binom{n}{2n_1}^r f_{n}(n_1)\binom{2(n-n_1)}{n-n_1}^u.
	\end{equation}
	Note that $\Delta^\prime\geq\Delta$, since $$\left\vert 2n_1-1-\frac{n}{2}\right\vert\leq\left\vert2n_1-2-\frac{n}{2}\right\vert$$ and $f_n(n_1)\geq f_n(n_1-1)$.
	Consequently, the inequality \eqref{3.19} implies \begin{equation}\label{3.17}
		X(n+1)-2X(n)\geq\Delta-\binom{n}{2n_1}^r f_{n}(n_1)\binom{2(n-n_1)}{n-n_1}^u.
	\end{equation}
	
	Write $n=4m+q$ with $m\in\Z^+$ and $q\in\{0,1,2,3\}$. We see that $n_1=m+1-\delta_{0,q}$. The proof of \eqref{3.15} then proceeds by cases, according to the value of $q$.
	
	If $q\in\{1,2\}$, then $$\left|\frac{n}{2}-2n_1\right|=\left| \frac{q}{2}+2\delta_{0,q}-2\right|=2-\frac{q}{2} \geq\frac{q}{2}=\left| \frac{q}{2}+2\delta_{0,q}\right| =\left| \frac{n}{2}-(2n_1-2)\right|,$$so we have $\binom{n}{2n_1-2}\geq\binom{n}{2n_1}$. Consequently, we have
	\begin{align}
		\begin{aligned}\label{3.18}
			\frac{1}{2}X(n)&\leq \sum_{\substack{0\leq l\leq n_0,\\l\notin\{n_1-1,n_1\}}} \binom{n}{2l}^rf_n(l)\binom{2(n-l)}{n-l}^u+\frac{1}{2}\sum_{l\in\{n_1-1,n_1\}}^{n_0} \binom{n}{2l}^rf_n(l)\binom{2(n-l)}{n-l}^u\\
			&\leq \sum_{\substack{0\leq l\leq n_0,\\l\notin\{n_1-1,n_1\}}} \binom{n}{2l}^rf_n(l)\binom{2(n-l)}{n-l}^u+\binom{n}{2n_1-2}^r f_{n}(n_1)\binom{2(n+1-n_1)}{n+1-n_1}^u\\
			&=X(n)+\Delta-\binom{n}{2n_1}^r f_{n}(n_1)\binom{2(n-n_1)}{n-n_1}^u.
		\end{aligned}
	\end{align}
	By \eqref{3.17} and \eqref{3.18}, we get that $$X(n+1)\geq\frac{3}{2}X(n).$$
	
	Now assume that $q=3$. We have $$\left|\frac{n}{2}-2n_1\right| =\frac{1}{2}=\left| \frac{n}{2}-(2n_1-1)\right| <\frac{3}{2}=\left| \frac{n}{2}-(2n_1-2)\right|,$$ so we get $$\binom{n}{2n_1-1}=\binom{n}{2n_1}>\binom{n}{2n_1-2}.$$
	Consequently, we have
	\begin{align}
		\begin{aligned}\label{3.20}
			\frac{1}{2}X(n)&\leq\sum_{\substack{0\leq l\leq n_0,\\l\notin\{n_1-1,n_1\}}} \binom{n}{2l}^r f_n(l)\binom{2(n-l)}{n-l}^u+\frac{1}{2}\sum_{l\in\{n_1-1,n_1\}} \binom{n}{2l}^rf_n(l)\binom{2(n-l)}{n-l}^u\\
			&\leq \sum_{\substack{0\leq l\leq n_0,\\l\notin\{n_1-1,n_1\}}} \binom{n}{2l}^r f_n(l)\binom{2(n-l)}{n-l}^u+\binom{n}{2n_1-1}^r f_{n}(n_1)\binom{2(n+1-n_1)}{n+1-n_1}^u\\
			&=X(n)+\Delta^\prime-\binom{n}{2n_1}^r f_{n}(n_1)\binom{2(n-n_1)}{n-n_1}^u.
		\end{aligned}
	\end{align}
	By \eqref{3.19} and \eqref{3.20}, we see that $$X(n+1)\geq\frac{3}{2}X(n).$$
	
	At last assume that $q=0$. Then we have $n=4m,n_0=2m$, and $n_1=m\geq2$. Note that
	\begin{equation*}
		\binom{n+1}{2n_1}^r=\binom{4m+1}{2m}^r=\left( \binom{4m}{2m}\frac{4m+1}{2m+1}\right) ^r\geq\frac{9}{5}\binom{4m}{2m}^r=\frac{9}{5}\binom{n}{2n_1}^r,
	\end{equation*}
	since $m\geq2$ and $r\geq1$. Thus, setting $$E=X(n+1)-\frac{9}{5}\binom{n}{2n_1}^rf_{n}(n_1)\binom{2(n+1-n_1)}{n+1-n_1}^u,$$ we have
	\begin{align}
		\begin{aligned}\label{eqE}
			E&\geq X(n+1)-\binom{n+1}{2n_1}^rf_{n+1}(n_1)\binom{2(n+1-n_1)}{n+1-n_1}^u\\
			&=\binom{2(n+1)}{n+1}^u+\sum_{1\leq k\leq n_0,k\neq n_1}\binom{n+1}{2k}^rf_{n+1}(k)\binom{2(n+1-k)}{n+1-k}^u\\
			&\geq\binom{2n}{n}^u+\sum_{1\leq k\leq n_0,k\neq n_1}\left( \binom{n}{2k}^r+\binom{n}{2k-1}^r\right)f_{n}(k)\binom{2(n+1-k)}{n+1-k}^u\\
			&\geq X(n)-\binom{n}{2n_1}^rf_{n}(n_1)\binom{2(n+1-n_1)}{n+1-n_1}^u+\sum_{k=1, k\neq n_1}^{n_0}\binom{n}{2k-1}^rf_{n}(k)\binom{2(n+1-k)}{n+1-k}^u.
		\end{aligned}
	\end{align}
	Recall that $B=X(n+1)-X(n)$. From \eqref{eqE}, we see that
	\begin{equation}\label{eqq32}
		B\geq\frac{4}{5}\binom{n}{2n_1}^rf_{n}(n_1)\binom{2(n+1-n_1)}{n+1-n_1}^u+\sum_{k=1, k\neq n_1}^{n_0} \binom{n}{2k-1}^rf_{n}(k)\binom{2(n+1-k)}{n+1-k}^u.
	\end{equation}
	We claim that
	\begin{equation}\label{eqq33}
		\sum_{k=1, k\neq n_1}^{n_0} \binom{n}{2k-1}^rf_{n}(k)\binom{2(n+1-k)}{n+1-k}^u\geq\sum_{\substack{0\leq l\leq n_0,\\l\notin\{n_1-1,n_1\}}} \binom{n}{2l}^r f_{n}(l)\binom{2(n-l)}{n-l}^u.
	\end{equation}
	In fact, the inequality \eqref{eqq33} follows from the inequalities $$\sum_{k=1}^{n_1-1} \binom{n}{2k-1}^rf_{n}(k)\binom{2(n+1-k)}{n+1-k}^u\geq\sum_{k=1}^{n_1-1} \binom{n}{2k-2}^rf_{n}(k-1)\binom{2(n+1-k)}{n+1-k}^u$$ and $$\sum_{k=n_1+1}^{n_0} \binom{n}{2k-1}^rf_{n}(k)\binom{2(n+1-k)}{n+1-k}^u\geq\sum_{k=n_1+1}^{n_0} \binom{n}{2k}^rf_{n}(k)\binom{2(n-k)}{n-k}^u.$$
	Observe that the term $$\binom{n}{2n_1}^rf_{n}(n_1)\binom{2(n+1-n_1)}{n+1-n_1}^u
	=\binom{4m}{2m}^rf_{4m}(m)\binom{6m+2}{3m+1}^u$$
	is not smaller than both $$\binom{n}{2n_1}^rf_{n}(n_1)\binom{2(n-n_1)}{n-n_1}^u=\binom{4m}{2m}^rf_{4m}(m)\binom{6m}{3m}^u$$
	and $$\binom{n}{2n_1-2}^rf_{n}(n_1-1)\binom{2(n-(n_1-1))}{n-(n_1-1)}^u=\binom{4m}{2m-2}^rf_{4m}(m-1)\binom{6m+2}{3m+1}^u.$$
	Hence, we have
	\begin{equation}\label{eqq34}
		\binom{n}{2n_1}^rf_{n}(n_1)\binom{2(n+1-n_1)}{n+1-n_1}^u\geq\frac{1}{2}\sum_{l\in\{n_1-1,n_1\}}\binom{n}{2l}^r f_{n}(l)\binom{2(n-l)}{n-l}^u.
	\end{equation}
	Therefore, by \eqref{eqq32}, \eqref{eqq33}, and \eqref{eqq34}, we deduce that
	\begin{equation}
		\begin{aligned}\label{3.21}
			B&\geq\frac{4}{5}\binom{n}{2n_1}^rf_{n}(n_1)\binom{2(n+1-n_1)}{n+1-n_1}^u+\sum_{\substack{0\leq l\leq n_0,\\l\notin\{n_1-1,n_1\}}} \binom{n}{2l}^r f_{n}(l)\binom{2(n-l)}{n-l}^u\\
			&\geq\frac{2}{5}\sum_{l\in\{n_1-1,n_1\}} \binom{n}{2l}^r f_{n}(l)\binom{2(n-l)}{n-l}^u+\frac{2}{5}\sum_{\substack{0\leq l\leq n_0,\\l\notin\{n_1-1,n_1\}}} \binom{n}{2l}^r f_{n}(l)\binom{2(n-l)}{n-l}^u\\
			&=\frac{2}{5}X(n),
		\end{aligned}
	\end{equation}
	which gives $$X(n+1)\geq\frac{7}{5}X(n).$$
	
	We have completed the proof of the claimed inequality \eqref{3.15} (and hence the theorem).
\end{proof}
\begin{remark}\label{Rem3.3}
	For every $n,r_1,r_2,s,t,u\in\Z_{\geq0}$ with $r_1+r_2\geq1$, define $$V(n,r_1,r_2,s,t,u)=\sum_{k=0}^n \binom{n}{ k}^{r_1}\binom{n}{2k}^{r_2}\binom{n+k}{k}^s\binom{2k}{k}^t\binom{2(n-k)}{n-k}^u.$$ Here we follow the convention $0^0=1$. Note that $(V(n,1,1,0,0,0))_{n}$ is the sequence of {\it quadrinomial coefficients} \seqnum{A005725}. When $r_1r_2=0$, i.e., when $r_1=0$ or $r_2=0$, the sequence $(V(n,r_1,r_2,s,t,u))_{n=1}^\infty$ is realizable by Remark \ref{Rem3.2} and Theorem \ref{Thm1.18}. Indeed, for all $r_1,r_2,s,t,u\in\Z_{\geq0}$ with $r_1+r_2\geq1$, an essentially same argument as the one given in the proof of Theorem \ref{Thm1.18} shows that the sequence $(V(n,r_1,r_2,s,t,u))_{n=1}^\infty$ is realizable. For shortening the paper, we omit the details.
\end{remark}
\begin{remark}\label{Rem3.4}
	Let $r_1,r_2,s,t,u$ be non-negative integers with $r_1+r_2\geq1$. Set $V(n)=V(n,r_1,r_2,s,t,u)$ for $n\geq1$. Clearly, we have $V(1)>0$. By checking the proof of Lemma \ref{Lem2.4}, it is easy to see that $(\mu*V)(n)>0$ for every $n\geq1$. According to Remark \ref{Rem3.3}, the sequence $(V(n))_{n}^\infty$ is realizable. Then $(\mu*V)(n)\geq n$ for all $n\geq1$. Thus, for every map $T:X\to X$ realizing $V$ and $n\in\Z^+$, there is at least one periodic orbit of $T$ with exact period $n$ in $X$.
\end{remark}
\begin{proof}[Proof of Theorem \ref{Thm1.20}]
	\leavevmode
	\begin{enumerate}[(1)]
		\item \label{Catalanpf} For an arbitrary prime number $p$, we have
		\begin{equation}\label{3.22}
			C(p)\equiv (p+1)C(p)=\binom{2p}{p}=\frac{(p+1)\cdots(p+p-1)}{(p-1)!}\times2\equiv 2\pmod{p},
		\end{equation}where the modulo is taken over the ring $\Z_p$. Thus, we have
		\begin{equation}\label{Catalanp}
			(\mu*C)(p)=C(p)-C(1)=C(p)-1\equiv 1\pmod{p}.
		\end{equation}
		Assume that $(C(n))_{n=1}^\infty$ is almost realizable. By \eqref{Catalanp}, we see that $p\mid{\rm Fail}(C)$. However, the prime number $p$ can be arbitrary large, contradicting the fact that $\Z^+\ni{\rm Fail}(C)<\infty$. Therefore, the sequence $(C(n))_{n=1}^\infty$ is not almost realizable.
		\item Let $p$ be an arbitrary odd prime number. By using \eqref{3.22} and Lemma \ref{Lem2.1}, we get
		\begin{align}
			\begin{aligned}\label{Motp}
				(\mu*M)(2p)&=M(2p)-M(p)-M(2)+M(1)\\
				&=\sum_{k=0}^{p} \binom{2p}{2k} C(k)-\sum_{k=0}^{\frac{p-1}{2}} \binom{p}{2k} C(k)-2+1\\
				&\equiv\binom{2p}{0} C(0)+\binom{2p}{2p} C(p)- \binom{p}{0} C(0)-2+1\\
				&=C(p)-1\\
				&\equiv1\pmod{p}.
			\end{aligned}
		\end{align}
		Assume that $(M(n))_{n=1}^\infty$ is almost realizable. By \eqref{Motp}, we deduce that $p\mid{\rm Fail}(M)$. Similar to (\ref{Catalanpf}), we see that $(M(n))_{n=1}^\infty$ is not almost realizable.
		\item Let $p$ be an arbitrary odd prime number. For an integer $p/2<k<p-1$, we have $$C(k)=\frac{1}{k+1}\binom{2k}{k}\equiv0\pmod{p}$$ by Lemma \ref{Lem2.1}. For an integer $0<k<p/2$, we have $\binom{p+k}{2k}\equiv0$ (mod $p$) by Lemma \ref{Lem2.1}. Note that $$C(p-1)=\frac{1}{p}\binom{2p-2}{p-1}=\frac{(p+1)\cdots(2p-2)(2p-1)}{(p-1)!}\frac{1}{2p-1}\equiv-1\pmod{p}.$$Then with \eqref{3.22}, we have
		\begin{align*}
			(\mu*S)(p)&=S(p)-S(1)=\sum_{k=0}^{p} \binom{p+k}{2k} C(k)-2\\
			&\equiv \binom{p}{0} C(0)+\binom{2p-1}{2p-2} C(p-1)+ \binom{2p}{2p}C(p)-2\\
			&\equiv C(p)-C(p-1)-1\\
			&\equiv2-(-1)-1=2\pmod{p}.
		\end{align*}
		Similar to (\ref{Catalanpf}), we see that $(S(n))_{n=1}^\infty$ is not almost realizable.
	\end{enumerate}
\end{proof}
\begin{remark}\label{Rem3.5}
	Although the sequences in Theorem \ref{Thm1.20} are not almost realizable, they all satisfy the sign condition. Clearly, we have $C(n+1)=\frac{2(2n+1)}{n+2}C(n)\geq2C(n)$, for every $n\geq1$. Hence, by Lemma \ref{Lem2.4}, the sequence $(C(n))_{n=1}^\infty$ satisfies the sign condition. Aigner \cite[Proposition 3]{Ai} proved that the sequence $(M(n))_{n=0}^\infty$ is log-concave, so the sign condition follows from $M(2)/M(1)=2$ and Lemma \ref{Lem2.4}. Xia and Yao \cite[Corollary 7]{XY} showed that the sequence $(s(n))_{n=0}^\infty$ of little Schr\"oder numbers is strictly log-concave, so the sign condition for $(s(n))_{n=1}^\infty$ follows from $s(2)/s(1)=3$ and Lemma \ref{Lem2.4}. Clearly, the sequence $(S(n))_{n=1}^\infty$ of large Schr\"oder numbers also satisfies the sign condition.
\end{remark}

\section{Acknowledgments}
We thank the anonymous referee for a careful reading of a first version of the manuscript and for numerous insightful remarks and suggestions that greatly improved the quality of the paper.

\bigskip
\hrule
\bigskip

\noindent 2020 {\it Mathematics Subject Classification}:
Primary 11B50; Secondary 37P35, 05A10, 11B83.

\noindent \emph{Keywords: } realizability, almost realizability, Ap\'ery number, Franel number, Domb number, central Delannoy number.

\bigskip
\hrule
\bigskip

\noindent (Concerned with sequences
\seqnum{A000032},
\seqnum{A000041},
\seqnum{A000045},
\seqnum{A000108},
\seqnum{A000110},
\seqnum{A000166},
\seqnum{A000172},
\seqnum{A000225},
\seqnum{A000364},
\seqnum{A000984},
\seqnum{A001003},
\seqnum{A001006},
\seqnum{A001067},
\seqnum{A001263},
\seqnum{A001850},
\seqnum{A002426},
\seqnum{A002445},
\seqnum{A002893},
\seqnum{A002895},
\seqnum{A005258},
\seqnum{A005259},
\seqnum{A005260},
\seqnum{A005725},
\seqnum{A006318},
\seqnum{A006953},
\seqnum{A053175},
\seqnum{A054783},
\seqnum{A062510},
\seqnum{A081085},
\seqnum{A122045}, and
\seqnum{A226158}.)


\begin{thebibliography}{10}
	
	\bibitem{Ai}
	M.~Aigner, Motzkin numbers, {\em European J. Combin.} {\bf 19} (1998), 663--675.
	
	\bibitem{Ap}
	A.~Ap\'ery, Irrationalit\'e de {$\zeta(2)$} et {$\zeta(3)$}, {\em Ast\'erisque} {\bf 61} (1979), 11--13.
	
	\bibitem{AKK}
	Nobuhiro Asai, Izumi Kubo, and H.-H. Kuo, Bell numbers, log-concavity, and log-convexity, {\em Acta Appl. Math.} {\bf 63} (2000), 79--87.
	
	\bibitem{BaSc}
	C.~Banderier and S.~Schwer, Why Delannoy numbers?, {\em J. Statist. Plann. Inference} {\bf 135} (2005), 40--54.
	
	\bibitem{BR}
	H.~W. Becker and J.~Riordan, The arithmetic of {B}ell and {S}tirling numbers, {\em Amer. J. Math.} {\bf 70} (1948), 385--394.
	
	\bibitem{Bell}
	E.~T. Bell, The iterated exponential integers, {\em Ann. of Math.} {\bf 39} (1938), 539--557.
	
	\bibitem{Beu}
	F.~Beukers, Another congruence for the {A}p\'ery numbers, {\em J. Number Theory} {\bf 25} (1987), 210--220.
	
	\bibitem{C}
	E.~Catalan, Sur les nombres de {S}egner, {\em Rend. Circ. Mat. Pal.} {\bf 1} (1887), 190--201.
	
	\bibitem{CCL}
	H.~H. Chan, S.~H. Chan, and Z.~Liu, Domb's numbers and {R}amanujan-{S}ato type series for {$1/\pi$}, {\em Adv. Math.} {\bf 186} (2004), 396--410.
	
	\bibitem{CX}
	W.~Y.~C. Chen and E.~X.~W. Xia, The 2-log-convexity of the {A}p\'ery numbers, {\em Proc. Amer. Math. Soc.} {\bf 139} (2011), 391--400.
	
	\bibitem{Domb}
	C.~Domb, On the theory of cooperative phenomena in crystals, {\em Adv. Phys.} {\bf 9} (1960), 149--361.
	
	\bibitem{D}
	T.~Do\v{s}li\'c, Log-balanced combinatorial sequences, {\em Int. J. Math. Math. Sci.} {\bf 4} (2005), 507--522.
	
	\bibitem{EPPW}
	G.~Everest, A.~J. van~der Poorten, Y.~Puri, and T.~Ward, Integer sequences and periodic points, {\em J. Integer Sequences} {\bf 5} (2002), \href{https://cs.uwaterloo.ca/journals/JIS/VOL5/Ward/ward2.html}{Article
		02.2.3}.
	
	\bibitem{Franel}
	J.~Franel, On a question of {L}aisant, {\em L'Interm\'ediaire des Math\'ematiciens} {\bf 1} (1894), 45--47.
	
	\bibitem{CT}
	C.~Helou and G.~Terjanian, On {W}olstenholme's theorem and its converse, {\em J. Number Theory} {\bf 127} (2008), 475--499.
	
	\bibitem{K}
	E.~Kummer, \"{U}ber die {E}rg\"{a}nzungss\"{a}tze zu den allgemeinen {R}eciprocit\"{a}tsgesetzen, {\em J. Reine Angew. Math.} {\bf 44} (1852),
	93--146.
	
	\bibitem{LF1}
	P.~G. Larcombe and D.~R. French, On the ``other'' {C}atalan numbers: a historical formulation re-examined, {\em Congr. Numer.} {\bf 143} (2000),
	33--64.
	
	\bibitem{LF2}
	P.~G. Larcombe and D.~R. French, A new generating function for the {C}atalan-{L}arcombe-{F}rench sequences: proof of a result by {J}ovovic, {\em Congr. Numer.} {\bf 166} (2004), 161--172.
	
	\bibitem{LW}
	J.~H.~van Lint and R.~M. Wilson, {\em A Course in Combinatorics}, Cambridge Univ. Press, 2nd edition, 2001.
	
	\bibitem{MS}
	A.~Malik and A.~Straub, Divisibility properties of sporadic {A}p\'ery-like numbers, {\em Res. Number Theory} {\bf 2} (2016), no. 5.
	
	\bibitem{MaoS}
	G.-S. Mao and M.~J. Schlosser, Supercongruences involving domb numbers and binary quadratic forms, {\em Rev. R. Acad. Cienc. Exactas F\'is. Nat. Ser. A Mat. RACSAM} {\bf 118} (2024), no. 11.
	
	\bibitem{Ma}
	A.~Mar\'oti, On elementary lower bounds for the partition function, {\em Integers} {\bf 3} (2003), \#A10.
	
	\bibitem{MiW}
	P.~Miska and T.~Ward, Stirling numbers and periodic points, {\em Acta Arith.} {\bf 201} (2021), 421--435.
	
	\bibitem{M}
	P.~Moss, {\em The arithmetic of realizable sequences}, PhD thesis, University of East Anglia, 2003.
	
	\bibitem{MoW}
	P.~Moss and T.~Ward, Fibonacci along even powers is (almost) realizable, {\em Fibonacci Quart.} {\bf 60} (2022), 40--47.
	
	\bibitem{Motzkin}
	T.~Motzkin, Relations between hypersurface cross ratios, and a combinatorial formula for partitions of a polygon, for permanent preponderance, and for non-associative products, {\em Bull. Amer. Math. Soc.} {\bf 54} (1948), 352--360.
	
	\bibitem{OS}
	R.~Osburn and B.~Sahu, A supercongruence for generalized {D}omb numbers, {\em Funct. Approx. Comment. Math.} {\bf 48} (2013), 29--36.
	
	\bibitem{P}
	Y.~Puri, {\em Arithmetic properties of periodic orbits}, PhD thesis, University of East Anglia, 2000.
	
	\bibitem{PW}
	Y.~Puri and T.~Ward, A dynamical property unique to the {L}ucas sequence, {\em Fibonacci Quart.} {\bf 39} (2001), 398--402.
	
	\bibitem{Schmidt}
	A.~L. Schmidt, Legendre transforms and {A}p\'ery's sequences, {\em J. Aust. Math. Soc.} {\bf 58} (1995), 358--375.
	
	\bibitem{OEIS}
	N.~J.~A. Sloane et~al., The {O}n-{L}ine {E}ncyclopedia of {I}nteger {S}equences, 2024. Available at \url{https://oeis.org}.
	
	\bibitem{Stanley}
	R.~P. Stanley, {\em Enumerative Combinatorics, Vol. 2}, Vol.~62 of {\em Cambridge Stud. Adv. Math.}, Cambridge Univ. Press, 1999.
	
	\bibitem{Stanley1}
	R.~P. Stanley, {\em Enumerative Combinatorics, Vol. 1}, Vol.~49 of {\em Cambridge Stud. Adv. Math.}, Cambridge Univ. Press, 2nd edition, 2011.
	
	\bibitem{S}
	A.~Straub, Gessel-Lucas congruences for sporadic sequences, {\em Monatsh. Math.} (2023), published online.
	\newblock Available at \url{https://doi.org/10.1007/s00605-023-01894-3}.
	
	\bibitem{Sul}
	R.~A. Sulanke, Objects counted by the central Delannoy numbers, {\em J. Integer Sequences} {\bf 6} (2003), \href{https://cs.uwaterloo.ca/journals/JIS/VOL6/Sulanke/delannoy.html}{Article 03.1.5}.
	
	\bibitem{T}
	J.~Touchard, Propri\'et\'es arithm\'etiques de certains nombres recurrents, {\em Ann. Soc. Sci. Bruxelles} {\bf 53A} (1933), 21--31.
	
	\bibitem{W}
	A.~J. Windsor, Smoothness is not an obstruction to realizability, {\em Ergodic Theory Dynam. Systems} {\bf 28} (2008), 1037--1041.
	
	\bibitem{XY}
	E.~X.~W. Xia and O.~X.~M. Yao, A criterion for the log-convexity of combinatorial sequences, {\em Electron. J. Combin.} {\bf 20} (4) (2013), \#P3.
	
	\bibitem{WS}
	W.~Xia and Z.-W. Sun, On congruences involving {A}p\'ery numbers, {\em Proc. Amer. Math. Soc.} {\bf 151} (2023), 3305--3315.
	
	\bibitem{Zagier}
	D.~Zagier, Integral solutions of {A}p\'ery-like recurrence equations, In {\em Groups and {S}ymmetries: {F}rom {N}eolithic {S}cots to {J}ohn {M}c{K}ay}, Vol.~47 of {\em CRM Proc. Lecture Notes}, pp. 349--366. Amer.
	Math. Soc., 2009.
	
\end{thebibliography}
\end{document}